\documentclass[10pt]{article}
\title{Minimal polynomial descriptions of polyhedra  and special semialgebraic sets}
\author{Gennadiy Averkov\footnote{Fakult\"at f\"ur Mathematik, Universit\"at Magdeburg, Universit\"atsplatz 2, D-39106 Magdeburg, email: averkov@ovgu.de} \ and Ludwig Br\"ocker\footnote{Mathematisches Institut, Universit\"at M\"unster, Einsteinstr. 62, D-48149, M\"unster, email: broe@math.uni-muenster.de}}

\usepackage{amsfonts,amsmath,amssymb,enumerate,amsthm,ifthen,graphpap,longtable,time,xspace,graphicx}
\usepackage[active]{srcltx}

\theoremstyle{plain}

\newtheorem{nn}{}[section]
\newtheorem{corollary}[nn]{Corollary}
\newtheorem{lemma}[nn]{Lemma}
\newtheorem*{lemma*}{Lemma}
\newtheorem{theorem}[nn]{Theorem}
\newtheorem{proposition}[nn]{Proposition}

\theoremstyle{definition}

\newtheorem{definition}[nn]{Definition}
\newtheorem{notationsremarks}[nn]{Notations and Remarks}
\newtheorem{remark}[nn]{Remark}

\numberwithin{equation}{section}

\newcommand{\aff}{\rmcmd{aff}}
\newcommand{\bd}{\rmcmd{bd}}
\newcommand{\calF}{\mathcal{F}}
\newcommand{\calU}{\mathcal{U}}
\newcommand{\cl}{\rmcmd{cl}}

\newcommand{\dotvar}{\,\cdot\,}
\newcommand{\eps}{\varepsilon}
\newcommand{\ext}{\rmcmd{ext}}
\newcommand{\intr}{\rmcmd{int}}
\newcommand{\mymark}[1]{{#1}^\ast}
\newcommand{\natur}{\mathbb{N}}
\newcommand{\overtwocond}[2]{\substack{{#1} \\ {#2}}}
\newcommand{\real}{\mathbb{R}}
\newcommand{\relint}{\rmcmd{relint}}
\newcommand{\rmcmd}[1]{\mathop{\mathrm{#1}}\nolimits}
\newcommand{\setcond}[2]{\left\{#1\,:\,#2\right\}}
\newcommand{\term}[1]{\emph{#1}}
\newcommand{\thmtitle}[1]{{\upshape(#1)}}
\newcommand{\sign}{\rmcmd{sign}}

\begin{document}
\maketitle

\begin{abstract}
We show that a $d$-dimensional polyhedron $S$ in $\real^d$ can be represented by $d$-polynomial inequalities, that is, $S = \setcond{x \in \real^d}{p_0(x) \ge 0, \ldots, p_{d-1}(x) \ge 0}$, where $p_0,\ldots,p_{d-1}$ are appropriate polynomials. Furthermore, if an elementary closed semialgebraic set $S$ is given by polynomials $q_1,\ldots,q_k$ and for each $x \in S$ at most $s$ of these polynomials vanish in $x$, then $S$ can be represented by $s+1$ polynomials (and by $s$ polynomials under the extra assumption that the number of points $x \in S$ in which $s$ $q_i$'s vanish is finite).
\end{abstract}

\newtheoremstyle{itsemicolon}{}{}{\mdseries\rmfamily}{}{\itshape}{:}{ }{}
\newtheoremstyle{itdot}{}{}{\mdseries\rmfamily}{}{\itshape}{.}{ }{}
\theoremstyle{itdot}
\newtheorem*{msc*}{2000 Mathematics Subject Classification} 

\begin{msc*}
  Primary: 14P05, 52B11, 14Q99;  Secondary: 52A20
\end{msc*}

\newtheorem*{keywords*}{Key words and phrases}

\begin{keywords*}
H\"ormander-{\L}ojasiewicz's Inequality; polyhedron; polynomial; polytope; semialgebraic set; stability index; Theorem of Br\"{o}cker and Scheiderer
\end{keywords*}

\section{Introduction}

Let $S \subseteq \real^d$ be a basic closed semialgebraic set, say $$S = \setcond{x \in \real^d}{p_1(x) \ge 0, \ldots, p_k(x) \ge 0} =: \{p_1\ge 0,\ldots,p_k \ge 0\},$$
where $p_i \in \real[X]$, $X := X_1,\ldots,X_d$. 

It is known since the eighties that one can choose, for the description of $S$, polynomials $p_1,\ldots,p_k$ such that $k \le \frac{d(d+1)}{2}$ (compare \cite{MR1137812}). Scheiderer \cite{MR1005003} gave examples showing that this bound is sharp. However, in these examples $S$ admits points $x$ where the local dimension of $S$ at $x$ is $m$ for all $1 \le m \le d$. So one might ask, if the bound for $k$ equals $d$ for sets $S$ of constant local dimension. In her diploma thesis A.~Pauluhn \cite{Pauluhn90} showed that for $d \in \{2,3\}$ equal dimensional basic closed sets can be characterized by at most $2$ and $4$ polynomials, respectively. All this holds true, if $\real$ is replaced by an arbitrary real closed field $R$. It seems that more is not known. Several authors (see \cite{Bernig98}, \cite{Groetschel-Henk-2003}, \cite{Bosse-Groetschel-Henk-2005}, \cite{BosseDiss}) looked at the case where $S$ is a polytope, which might be interesting for applications (see also \cite{Henk-2007} for a survey on this topic). Also, they tried to find effective computations for suitable polynomials $p_i$ with $i=1,\ldots, s$, satisfying $S = \{p_1 \ge 0,\ldots, p_s \ge 0\}$ starting from the description $S= \{l_1 \ge 0, \ldots, l_k \ge 0\}$, where $l_i$ are linear forms (i.e., polynomials of degree at most one) and $k$ might be very  large. Let $s=s(S)$ be the minimal possible value as above. One achieved the bound $s \le 2d-1$ in \cite{Bosse-Groetschel-Henk-2005}. In \cite{Groetschel-Henk-2003} one noticed that $s \ge d$ for polytopes and in \cite{Bosse-Groetschel-Henk-2005} one conjectured that $s=d$ for $d$-dimensional polytopes. The equality $s = d$ was shown in \cite{Bernig98} for polygons,  in \cite{Averkov-Henk-2009} for simple polytopes  and in \cite{Averkov-Henk-2009b} for three-dimensional polyhedra. The following two theorems are the main results of the manuscript. Theorem~\ref{av:thm} below presents a short proof of a generalization of the result for simple polytopes from \cite{Averkov-Henk-2009}. Theorem~\ref{every polyhedron is representable} computes $s$ for all $d$-dimensional polyhedra.

\begin{theorem}  \label{av:thm} Let $S \subseteq \real^d$ be bounded and basic closed, say $S= \{q_1 \ge 0,\ldots,q_k \ge 0\}$, where $q_1,\ldots,q_k \in \real[X]$.  Let $s \in \natur$ be such that for each $x \in S$ there are at most $s$ polynomials $q_i$ among $q_1,\ldots,q_k$ where $q_i(x) =0.$ Then the following statements hold.
\begin{enumerate}[a)] \item \label{av:thm:a}  $S=\{p_1,\ldots,p_{s+1} \ge 0\}$ for suitable $p_1,\ldots,p_{s+1} \in \real[X].$ 
\item \label{av:thm:b} If there are only finitely many points $x_1,\ldots,x_m \in S$ where exactly $s$ polynomials $q_i$ vanish, then $S = \{p_1 \ge 0,\ldots,p_s \ge 0\}$ for suitable $p_1,\ldots,p_s \in \real[X].$ 
\end{enumerate}
\end{theorem} 

In view of Theorem~\ref{av:thm}, every simple $d$-dimensional polytope can be represented by $d$ polynomial inequalities. The above statement is also covered by the following

\begin{theorem} \label{every polyhedron is representable}
 Let $S$ be a $d$-dimensional polyhedron in $\real^d$. Let $k$ be the maximal dimension of an affine space contained in $S.$ Then there exist $d-k$ polynomials $p_0,\ldots,p_{d-k-1}$ such that $S = \{p_0 \ge 0,\ldots,p_{d-k-1} \ge 0\}$. Furthermore, $S$ cannot be represented by less than $d-k$ polynomials.
\end{theorem}

In Section~\ref{separation} we present several separation theorems for semialgebraic sets. In Section~\ref{sect:simple:semialg} we prove Theorem~\ref{av:thm}. Finally, in Section~\ref{sect:polytopes} we prove Theorem~\ref{every polyhedron is representable}.

As in the above mentioned papers dealing with polynomial representations of polytopes, our work is semi-effective. That means, one has to check sequences of first order statements. As a consequence, in our theorems we can only compute the polynomials $p_i$ but one cannot bound their degrees in terms of the complexity of the ``input polynomials''. Equivalently, what we do below does not work over any real closed field. 

We shall use the following notations. The Euclidean norm of $\real^d$ is denoted by $\|\dotvar\|$. For the Euclidean topology of $\real^d$ we denote by $\cl$, $\intr$, $\bd$ the closure, interior and boundary, respectively.  We write $\cl^Z$ for the \term{Zariski closure}. Furthermore, the notations $\dim,$ $\aff,$ and $\relint$ stand for dimension, affine hull, and relative interior, respectively. For $x \in\real^d$ and $\rho>0$ let $B(x,\rho):= \setcond{y \in \real^d}{\|x-y\| \le \rho}$ and $\calU(x,\rho) := \setcond{y \in \real^d}{\|x-y\| < \rho}$. 

We write as before $\{p_1 \ge 0,\ldots,p_k \ge 0\}$ instead of $\setcond{x \in \real^d}{p_1(x) \ge 0,\ldots,p_k(x) \ge 0}$ for polynomials $p_1,\ldots,p_k \in \real[X]$. Similarly we write $\{p_1>0,\ldots,p_k>0\}$ and $\{p_1=0,\ldots,p_k=0\}$. Note that $V:=\{p_1=0,\ldots,p_r=0\} = \{q=0\}$ for $q=p_1^2 + \cdots + p_r^2$. $q$ is also called \term{positive polynomial} for the algebraic set $V$.

\section{Separation} \label{separation}

We shall use standard inequalities on continuous semialgebraic functions, see \cite[Section~2.6]{MR1659509}.

\begin{theorem} \label{more than loj}
 Let $S \subseteq \real^d$ be a closed semialgebraic set and let $f, g, h$ be continuous semialgebraic functions on $S$ with $\{g=0\} \cap S \subseteq \{f=0\} \cap S.$ Then there exists a positive polynomial $p$ and $N \in \natur$ such that $|f^N h| \le |p g|$ on $S$.
\end{theorem}

The version for $h(X)=1$ follows from the results in \cite[Section~2.6]{MR1659509}. The version for $h(X)=1$ obviously implies the version for a general $h$. 
For the special case that $S$ is bounded and $h(X)=1$ we may obviously choose $p$ to be a constant, which yields the well-known \term{H\"ormander-{\L}ojasiewicz Inequality}. The polynomial $p$ can be chosen to have a specific form $p(X) = (1+ \|X\|^2)^M$ with $M \in \natur.$ Some consequences of Theorem~\ref{more than loj} (and, more specifically, H\"ormander-{\L}ojasiewicz's Inequality) will be useful in our subsequent derivations. 

\begin{proposition} \label{f^N>g}
 Let $A$ be an unbounded, closed semialgebraic set and $f, g$ be continuous semialgebraic functions on $A$ such that $f > 0$, $g \ge 0$ and $f(x) \rightarrow \infty$ as $x \in A$ and $x \rightarrow \infty$. Then there exists $\gamma > 0$ and $N \in \natur$ such that $\gamma f^N \ge g$ on $A$.
\end{proposition}
\begin{proof}
	Let $A = A_0 \cup A_1$, where $A_0, A_1$ are semialgebraic, closed, $A_0$ is bounded, $o \not\in A_1$ and $f \ge 2$ on $A_1$.

	By Theorem~\ref{more than loj}, $g \le \alpha (1+ \|X\|^2)^M$ on $A$ for appropriate $\alpha >0$ and $M \in \natur$. The inequality 
	\begin{equation} \label{f^N grows fast}
	 f^N \ge \alpha (1+ \|X\|^2)^M
	\end{equation}
	is fulfilled on $A_1$ if 
	\[
	 	f\left(\frac{X}{\|X\|^2}\right)^{-N} \le \frac{\|X\|^{2M}}{\alpha (\|X\|^2+1)^M}
	\]
	for all $X$ with $\frac{X}{\|X\|^2} \in A_1$.
	Let $\Tilde{A}_1:= \{o\} \cup \setcond{x \in \real^d \setminus \{o\}}{\frac{x}{\|x\|^2}  \in A_1}$ and $\tilde{f}$ a semialgebraic function on $\Tilde{A}_1$ given by 
	\[
	 	\Tilde{f}(x):=\begin{cases}
	 	  f\left(\frac{X}{\|X\|^2}\right)^{-1} & \mbox{if} \ x \in \Tilde{A}_1 \setminus \{o\}, \\
			0 & \mbox{if} \ x=0.
	 	\end{cases}
	\]
	Clearly, $\Tilde{A}_1$ is bounded and closed. Thus, for verification of \eqref{f^N grows fast} we need to show 
	\[ 
		\Tilde{f}^N \le \frac{\|X\|^{2M}}{\alpha (\|X\|^2+1)^M} \qquad \mbox{on $\Tilde{A}_1$}.
	\]
	The existence of $N$ satisfying the above relation follows from {\L}ojasiewicz's inequality (by taking into account that $\tilde{f} \le \frac{1}{2}$ on $\Tilde{A}_1$).
	The assertion follows by defining a $\gamma>1$ such that $\gamma f^N \ge g$ on $A_0$.
\end{proof}

\begin{definition} Let $S, T \subseteq \real^d$ be semialgebraic sets and let $p \in \real[X].$ We say that $p$ \term{separates} $S$ from $T$ if $p\ge 0$ on $S,$ $p \le 0$ on $T$, and $\{p=0\} \cap (S\cup T) \subseteq \cl^Z(S \cap T).$ 
\end{definition}

The polynomials from Lemma~\ref{lambda:mu:lem} will be frequently used in our constructions. 

\begin{lemma} \label{lambda:mu:lem}
Let $0< \delta < \rho$ and $m \in \natur$. Then the following two statements hold true.
\begin{enumerate}[a)]
  \item \label{kappa:part} There exists a polynomial $\kappa=\kappa_{\delta,\rho,m} \in \real[t]$ such that
    \begin{enumerate}[i)]
    	\item $\kappa \ge 0$ on $[-\rho,\rho]$ and $\{\kappa = 0\} \cap [-\rho,\rho] = \{0\}$,
    	\item $\kappa \le \frac{1}{4^m}$ on $[-\rho,0]$,
    	\item $\kappa \ge 2$ on $[\delta,\rho]$,
    	\item $\kappa \le 3$ on $[0,\rho]$.
    \end{enumerate}
	\item \label{mu:part} There exists a polynomial $\mu :=\mu_{\delta,\rho} \in \real[t]$ such that 
	\begin{enumerate}[i)]
		\item $\mu>0$ on $[-\rho,\infty[,$ 
		\item $\mu<\frac{1}{2}$ on $[-\rho,0],$ 
		\item $\mu>2$ on $[\delta,\infty[.$ 
	\end{enumerate} 
	In particular, $\mu(t) \rightarrow \infty$ for $t \rightarrow \infty$
\end{enumerate}
\end{lemma}
\begin{proof}
  \ref{kappa:part}) Consider a continuous function $\phi : [-\rho,\rho] \rightarrow \real$ such that $\frac{1}{4^{m+1}} \le t^2 \phi(t) \le \frac{1}{2 \cdot 4^m}$ for $t \in [-\rho,0]$, $t^2 \phi(t) \ge 2 + \frac{1}{3}$ for $t \in [\delta,\rho]$ and $t^2 \phi(t) \le 3 - \frac{1}{3}$ for $t \in [0,\rho]$. Now by explicit Stone-Weierstrass Approximation (cf. \cite[Theorem~8.8.5]{MR1659509}) of $\phi(t)$ by a polynomial $\pi(t)$ we get $\kappa(t) = t^2 \pi(t)$.
  
	\ref{mu:part}) We set 
	$$
		\mu_{\delta,\rho}(t) := \left( \frac{t+ 2 \rho - \delta /2}{2 \rho}\right)^{2k},
	$$
	where $k \in \natur$ is sufficiently large. 
\end{proof}

\begin{proposition} \label{sep:disjoint} \thmtitle{Separation of disjoint closed sets}
 Let $S, T \subseteq \real^d$ be closed semialgebraic sets such that $S \cap T = \emptyset.$ Let $S$ be compact and basic closed, say $S = \{f_1 \ge 0,\ldots,f_k \ge 0\}$, where $f_1,\ldots,f_k \in \real[X]$. Then there exists a polynomial $p \in \real[X]$ which separates $S$ from $T.$ 
\end{proposition}
\begin{proof}
We define the polynomial mapping $F(X):=(f_1(X),\cdots, f_k(X))$ from $\real^d$ to $\real^k$. Then $F(S) \subseteq \real_{\ge 0}^k$ and $F(T) \subseteq \real^k \setminus \real_{\ge 0}^k$ are compact resp. closed semialgebraic sets. Choose $\rho>0$ such that $F(S) \subseteq [0,2 \rho]^k$. We define the polynomial $g_m(Y):= \rho^{2 m} k + \frac{1}{m} - \sum_{i=1}^k (Y_i-\rho)^{2 m}$ in indeterminates $Y:=Y_1,\ldots,Y_k$, where $m \in \natur$. The semialgebraic set $G_m:= \{g_m \ge 0\} \subseteq \real^k$ approximates $[0,2 \rho]^k$ with any given precision (in the Hausdorff metric), as $m \rightarrow \infty$. Moreover, $[0,2\rho]^k$ is contained in the interior of $G_m$ for every $m \in \natur$. Consequently, $G_m$ is disjoint with $F(T)$ if $m$ is sufficiently large. Hence we may define $p:=g_m \circ F$ with $m$ sufficiently large.
\end{proof}

\begin{remark}
 Proposition~\ref{sep:disjoint} still holds, if $S$ is not necessarily basic. Also, it can be shown directly by Stone-Weierstrass Approximation, but the way we did it is more constructive.
\end{remark}

\begin{proposition} \label{merge:loc:glob} \thmtitle{Globalizing a local separator}
Let $S \subseteq \real^d$ be bounded and closed semialgebraic set and let $T \subseteq \real^d$ be a closed semialgebraic set. Let $r \in \real[X]$ be such that $r \ge 0$ on $\real^d$, $S \cap T = \{r =0 \}$ and $r(X) \rightarrow \infty$ as $\|X\| \rightarrow \infty$. Let $f \in \real[X]$ be such that $f$ separates $S \cap U$ from $T \cap U$ for an open, semialgebraic set $U$ with $S \cap T \subseteq U$. Then there exists a polynomial $p$ separating $S$ from $T$.
\end{proposition}
\begin{proof}
Since $r(x) \rightarrow \infty$ as $\|x\| \rightarrow \infty$, the set $\{r \le \eps\}$ is bounded for every $\eps \ge 0$. If $0 < \eps \le \frac{1}{2}$ is small enough, then $\{r \le \eps\} \subseteq U$. In view of Proposition~\ref{sep:disjoint}, there exists a polynomial $g$ that separates $S \cap \{r \ge \frac{\eps}{4}\}$ from $T \cap \{r \ge\frac{\eps}{4}\}$. We define $p:=f+ \left(\frac{2 r}{\eps}\right)^N g$ with $N \in \natur$ to be chosen below. In view of \L{o}jasiewicz's inequality we have $\frac{1}{2} |f| \ge \left|\frac{2 r}{\eps}\right|^N |g|$ on $( S \cup T) \cap \{r \le \frac{\eps}{4}\}$ for all sufficiently large $N$. On the other hand, by Proposition~\ref{f^N>g}, $\frac{1}{2} \left| \frac{2 r}{\eps}\right|^N |g| \ge |f|$ on $(S \cup T) \cap \{r \ge \eps\}$ for all sufficiently large $N$. Thus, $p$ separates $S$ from $T$ if $N$ is sufficiently large.
\end{proof}

\begin{proposition} \label{merge local sep}
 \thmtitle{Merging local separators} Let $S, T \subseteq \real^d$ be closed semialgebraic sets and let $x_1,\ldots,x_m \in \real^d$. Assume that there exists $\rho>0$ such that for every $i \in \{1,\ldots,m\}$ one has $\{x_i\} = S \cap T \cap B(x_i,\rho)$ and one can find a polynomial $p_i$ separating $S \cap B(x_i,\rho)$ from $T \cap B(x_i,\rho)$. Then there exist $\rho>0$ and $p \in \real[X]$ such that $p$ separates $S \cap U$ from $T \cap U$ for $U:= \bigcup_{i=1}^m B(x_i,\rho)$.
\end{proposition}
\begin{proof}
	Let $r_i:=\|X-x_i\|^2$. Choose $\rho>0$ to be small enough and a $\delta$ such that $\rho < \delta < \|X_i-X_j\| - \rho$ for $1 \le i < j \le m$. Then, by {\L}ojasiewicz's inequality, for a sufficiently large $N \in \natur$ one has $\left(\frac{r_i}{\delta^2}\right)^N \le \frac{1}{2} |p_i|$ on $(S \cup T) \cap B(x_i,\rho)$ and $p_i + \left(\frac{r_i}{\delta^2}\right)^N > 0$ on $B(x_j,\rho)$ for $1 \le i < j \le m$. Thus, we may choose $$p:= \prod_{i=1}^m \left( p_i + \left(\frac{r_i}{\delta^2}\right)^N\right).$$
\end{proof}

The following proposition is a straightforward consequence of Propositions~\ref{merge:loc:glob} and \ref{merge local sep}.

\begin{proposition} \label{sep:finite:intersect} \thmtitle{Separation of sets with finite intersection} Let $S, T \subseteq \real^d$ be semialgebraic sets, $T$ closed, $S$ basic closed and bounded, such that $S \cap T$ is finite, say $S \cap T = \{x_1,\ldots, x_m\}.$ Then there exists a polynomial $p$ which separates $S$ from $T$ if and only if this holds locally, that means: There exists $\rho>0$ such that for $i \in \{1,\ldots,m\}$ there exists a polynomial $p_i$ which separates $B(x_i,\rho) \cap S$ from $B(x_i,\rho) \cap T.$ 
\end{proposition} 

\begin{remark} 
	The following generalization of  Proposition~\ref{sep:finite:intersect} holds true. Let $S, T \subseteq \real^d$ be semialgebraic, $S$ compact, $T$ closed. Then there exists a polynomial $p$ which separates $S$ from $T$ if and only if for any finite set of points $X$ there exists $\rho>0$ and a polynomial $q$ such that  $q$ separates $S \cap U$ from $T \cap U$ for $U:=\left( \bigcup_{x \in X} B(x,\rho) \right)$. This follows from Proposition~6.10 (Chapter~VI)  in \cite{MR1393194}, whose proof, however, is not at all constructive.
\end{remark} 

\begin{proposition} \label{S:T1:T2} Let $S, T_1, T_2 \subseteq \real^d$ be semialgebraic sets, where $S$ is basic closed, $T_1$ is compact and $T_2$ is closed. Let $h$ be a non-negative polynomial with $\{h=0\} = \cl^Z(S \cap T_1)$ and let $f \in \real[X]$ be an arbitrary polynomial. Assume that 
\begin{enumerate}[a)]
	\item \label{08.07.28,11:48} $S \cap T_2 = \emptyset,$ 
	\item \label{08.07.28,11:49} $\intr (S \cap T_1)=\emptyset$,
	\item \label{08.07.28,11:50} $h>0$ on $T_2$ and, for $x \in T_2,$ $h(x) \rightarrow \infty$, as $\|x\| \rightarrow \infty.$
	\item \label{08.07.28,11:51} $\{f =0\} \cap T_1 \subseteq S \cap T_1.$ 
\end{enumerate} 
	Then there exists $p \in \real[x]$ such that 
\begin{enumerate}[i)]
	\item \label{08.07.28,11:52} $p \ge f$ on $S,$ 
	\item \label{08.07.28,11:53} $\sign p \le \sign f$ on $T_1,$ 
	\item \label{08.07.28,11:54} $p< 0$ on $T_2.$
\end{enumerate}
\end{proposition}
\begin{proof}
	By assumption \ref{08.07.28,11:50}) there exists $\alpha>0$ such that $h\ge \alpha$ on $T_2.$ Choose $\delta$ with $0 < 2 \delta < \alpha$ and $\delta < \rho.$ We introduce the polynomial $\mu=\mu_{\delta,\rho}$ as in Lemma~\ref{lambda:mu:lem}\ref{mu:part}).  Let $T_0 := T_1 \cap \{h \ge \delta \}.$ Then $S \cap T_0 = \emptyset.$ Taking into account assumption \ref{08.07.28,11:48}) and applying Proposition~\ref{sep:disjoint}, we can find a polynomial $q \in \real[X]$ with $q>0$ on  $S$ and $q<0$ on $T_0 \cup T_2.$ Now, set $p:= f + q \cdot h^l \cdot [\mu \circ (h-\delta)]^{l+m}$, where $l, m \in \natur$ are sufficiently large.
\end{proof} 

\section{Small polynomial description of special semialgebraic sets} \label{sect:simple:semialg}

If $a_1,\ldots,a_k \in \real$, let $\sigma_i(a_1,\ldots,a_k)$ denote the $i$th \term{elementary symmetric} function of $a_1,\ldots,a_k$, that is, 
$$\sigma_i(a_1,\ldots,a_k) := \sum_{1 \le j_1 < \ldots < j_i \le k} a_{j_1} \cdots a_{j_i}.$$ 
The basic observation is the following. 

\begin{lemma} \label{el:sym:lemma} 
	Let $\rho > 0$ and $s, k \in \natur$ with $s \le k.$ Then there exists $\eps>0$ such that for all $a_1,\ldots,a_k \in \real$ with $|a_1| \le \eps, \ldots, |a_s| \le \eps$ and $a_{s+1} \ge \rho, \ldots, a_k \ge \rho$ one has $a_1,\ldots,a_s \ge 0$ if and only if $$\sigma_{k-s+1}(a_1,\ldots,a_k) \ge 0, \ldots, \sigma_{k}(a_1,\ldots,a_k) \ge 0.$$ 
\end{lemma}
\begin{proof} 
	The necessity is trivial. Let us prove the sufficiency. Consider the polynomial 
	$$
		\prod_{i=1}^k (t+a_i) =\sum_{i=1}^k \sigma_i(a_1,\ldots,a_k) t^{k-i} + t^k
	$$
	For $i \le k-s$ there is a summand in the definition of $\sigma_i(a_1,\ldots,a_k)$ without factors belonging to $\{a_1,\ldots,a_s\}$. The summands containing a factor of $\{a_1,\ldots,a_s\}$ tend to zero for $\eps \rightarrow 0.$ Hence for all sufficiently small $\eps>0$ depending on $k, s, \rho$ and for $i \le k-s$ one has $\sigma_i(a_1,\ldots,a_k) >0.$ One has $\sigma_i(a_1,\ldots,a_k) \ge 0$ for $i > k-s,$ by the assumption. Consequently, the polynomial $\prod_{i=1}^k (t+a_i)$ has no positive roots, which implies that $a_1 \ge 0,\ldots,a_k \ge 0.$ 
\end{proof} 

\begin{proof}[Proof of Theorem~\ref{av:thm}] 
	\ref{av:thm:a}) Let $p_i := \sigma_{k-s+i}(q_1,\ldots,q_k)$ for $i \in \{1,\ldots,s\}.$ Then $S \subseteq \{p_1 \ge 0,\ldots,p_s \ge 0\}.$ Moreover, by Lemma~\ref{el:sym:lemma}, for each $x \in \bd S$ there is a neighborhood $B(x,\eps)$ such that $S \cap B(x,\eps) = \{p_1 \ge 0,\ldots,p_s \ge 0\} \cap B(x,\eps).$ That means 
	\begin{equation} \label{S:cup:T:repr}
		\{p_1 \ge 0,\ldots,p_s \ge 0\} = S \cup T,
	\end{equation} 
	 where $T \subseteq \real^d$ is semialgebraic, closed and 
	\begin{equation} \label{T:S:disj}
		T \cap S = \emptyset.
	\end{equation}
	 Now, according to Proposition~\ref{sep:disjoint}, we can choose $p_{s+1} \in \real[X]$ which separates $S$ from $T.$ 

	\ref{av:thm:b}) We take $p_1,\ldots,p_{s}$ as before. Let $\{x_1,\ldots,x_m\}$ be the set of all points in $S$ where exactly $s$ polynomials $q_i$ vanish. We see that $R := ( \{p_1 \ge 0,\ldots, p_{s-1} \ge 0\} \setminus S) \cup \{x_1,\ldots,x_m\}$ is closed. Clearly, $R \cap S = \{x_1,\ldots,x_m\}.$ Moreover, for each $x_i$ there is a ball $B(x,\eps)$ such that $p_s$ separates $B(x_i,\eps) \cap S$ and $B(x_i,\eps) \cap R$. Now, according to Proposition~\ref{sep:finite:intersect}, we can modify $p_s$ to a polynomial separates $S$ from $R.$ 
\end{proof} 

\begin{remark}
Analogous semi-effective results can also be obtained for basic open sets by using similar methods.
\end{remark}

\section{Minimal description of polyhedra} \label{sect:polytopes}

A subset of $\real^d$ is said to be a \term{polyhedron} if it is the intersection of finitely many closed halfspaces. Bounded polyhedra are called \term{polytopes}. For background information on polytopes and polyhedra we refer to \cite{Ziegler-book-1995}. By $\calF(P)$ we denote the set of all faces of $P.$ By $\calF_k(P)$ we denote the set of all $k$-dimensional subfaces of  $P.$ If the choice of $P$ is clear we merely write $\calF_k$ and $\calF.$ Polytopes of dimension $d$ are called \term{$d$-polytopes}.  Faces of dimension $0$ and $\dim P-1$ are called \term{vertices} and \term{facets}, respectively. We introduce the \term{$k$-skeleton} (also called the set of all \term{$k$-extremal points}) of $P$ by $$\ext_k P := \bigcup_{F \in \calF_k(P)} F.$$ We also write $\ext P:=\ext_0 P$. 

\begin{notationsremarks} \label{rem:polytopes} Let $S \subseteq \real^d$ be a \emph{polyhedron}.
\begin{enumerate}[a)]
	\item For a $k$-face $F$ of $S$ we fix a degree-one polynomial $l_F$ such that $F = \{l_F = 0\} \cap P$ and $l_F \ge 0$ on $S.$ 
	Then 
	$$D_k(S):= \bigcap_{F \in \calF_k} \{l_F \ge 0\}$$ 
	is called a \emph{$k$-support} of $S.$ The $k$-support depends on the choice of $l_F$'s. We have $D_{d-1}(S) = S.$ 
	\item \label{08.07.29,15:59} For $l < k$ every $l$-face of $S$ is also an $l$-face of $D_k(S)$. 

	\item If $S$ is compact, $D_k(S)$ is not compact in general, unless $l_F$'s are chosen in a suitable way.

	\item For a vertex $x \in S$ we set 
	\begin{align*}
		S_x := \bigcap_{\overtwocond{F \in \calF_{d-1}(S)}{x \in F}} \{l_F \ge 0\} & & \mbox{and} & & \mymark{S}_x := \bigcap _{\overtwocond{F \in \calF_{d-1}(S)}{x \not\in F}} \{l_F > 0\}. 
	\end{align*}
	Notice that $P_x$ is closed and $\mymark{P}_x$ is open. 
	\item \label{chain:of:k:supports} Now let $S$ be a polytope. Then there exists a sequence of $k$-supports $D_{-1}(S),\ldots, D_{d-1}(S)$ with 
	$$
		S = D_{d-1}(S) \subseteq \cdots \subseteq D_0(S) \subseteq D_{-1}(S) = \real^d.
	$$
	and $D_k(S) \setminus \ext_{k-1} S \subseteq \intr (D_{k-1})$ for $k=0,\ldots,d-1$. 
	\item For $k=0,\ldots,d-1$ one has 
	\begin{enumerate}[i)]
		\item \label{08.07.29,16:25} $D_k(S) = \bigcap_{x \in \ext S} D_k(S)_x.$ 
		\item \label{08.07.29,16:26} There is a compact semialgebraic set $R_k$ such that $$S \subseteq \intr(R_k) \subseteq R_k \subseteq \bigcup_{x \in \ext S} \mymark{D_k(S)}_x.$$
	\end{enumerate}
	Here i) holds, since each facet of $D_k(S)$ contains a vertex of $S,$ and ii) follows from the inclusion 
	\begin{equation} \label{08.09.17,17:25}
		S \subseteq \bigcup_{x \in \ext S} \mymark{D_k(S)}_x.
	\end{equation} 
	Let us show \eqref{08.09.17,17:25}. Take an arbitrary $y \in S.$ Then there exists a unique face $G$ with $y \in \relint G.$ Let $x$ be any vertex of $G$ and let $F$ be a facet with $x \not\in F.$ Then $y \not\in F,$ since otherwise we would have $\relint G \cap F \ne \emptyset,$ which implies $G \subseteq F$ and by this $x \in F,$ a contradiction. Hence $l_F(y)>0.$  This yields \eqref{08.09.17,17:25}. 
\end{enumerate} 
\end{notationsremarks}

\begin{proposition} \label{main proposition}
 For $k=0,\ldots,d-1$ there is a polynomial $p_k$ such that 
\begin{enumerate}[a)]
 \item $p_k \ge 0$ on $S$
\item \label{main prop cond 2} $p_k \le 0$ on $D_{k-1}(S) \setminus \intr (D_k(S))$
\item \label{main prop cond 3} $\{p_k = 0\} \cap (D_{k-1}(S) \setminus \intr(D_k(S))) \subseteq S$
\end{enumerate}
\end{proposition}

In \ref{main prop cond 2}) and \ref{main prop cond 3}) we could also replace $D_{k-1}(S) \setminus \intr (D_k(S))$ by $D_{k-1}(S) \setminus D_k(S)$, but for the proof by induction, which we give below, it is better to have it this way. 

\begin{proof}[Proof of Proposition~\ref{main proposition}] The proof is by induction on $d$. For starting the induction argument we consider the cases $k=0$ and $k=d-1$ separately, which is done in Steps 1 and 2 of the proof. In the remaining steps we apply the inductive assumption to all vertex figures of $S$. Appropriately combining the polynomials associated to the vertex figures we generate polynomials associated to $S$. In Step~3 for each vertex of $S$ we construct  a polynomial satisfying a)-c) in a small neighborhood of that vertex. We combine these polynomials in Step 6, thus getting a polynomial $r_k$ and show in Step 7 that $r_k$ fulfills the conditions a),b),c) locally, that means in a set $Q_{k-1}$ which we get by restriction to $R_k$. This is the main step. It uses decompositions of $Q_{k-1}$ which are explained in Step 4 and 5. Finally, in Step 8 we globalize $r_k$ in order to get $p_k$.

\emph{Step~1: $k=0$.} We need to show the existence of $p_0$ such that $p_0 \ge 0$ on $S$, $p_0 \le 0$ on $\real^d \setminus \intr D_0(S)$ and $\{p_0=0\} \cap (\real^d \setminus \intr D_0(S)) \subseteq S$. We use Proposition~\ref{sep:finite:intersect}, where $T:= \real^d \setminus \intr(D_0(S))$. Then $S \cap T = \ext S $ is a finite set, and a local separation of $S$ from $T$ around each $x \in S \cap T$ can be achieved. 

\emph{Step~2: $k=d-1$.} Let $\mathcal{F}_{d-1}(S)=\{F_1,\ldots,F_r\}$. We define $p_{d-1} := \prod_{i=1}^r l_{F_i}$. Clearly, $p_{d-1} \ge 0$ on $S$. Now let 
$$S_i:= \{l_{F_i} \le 0\} \cap \bigcap_{\overtwocond{j=1,\ldots,r}{j \ne i}} \{l_{F_j} \ge 0\}.$$
Then by construction we have $D_{d-2}(S) \subseteq \bigcup_{i=1}^r S_i$, $p_{d-1} < 0$ on $(\bigcup_{i=1}^r S_i) \setminus S$ and $p_{d-1}=0$ on $\bd S$. This proves b) and c).

The case $d=1$ is trivial. Cases 1 and 2 yield the assertion for $d=2$. Now assume that $d \ge 3$.

\emph{Step~3: Construction of a local solution $q_{i,k}$.} Let $\ext S = \{x_0,\ldots,x_m\}$. For small $\eps>0$ the hyperplane $\{l_{x_i} = \eps\}$ intersects $S$, say $S^i := \{l_{x_i}= \eps\} \cap S$, where we choose $\eps$ such that $l_{x_i}(x_j) > \eps$ for $j \ne i$. Also, we set $D_{i,k-1} = \{l_{x_i} = \eps\} \cap D_k(S)$. So, we can use the inductive assumption to the $(d-1)$-dimensional polytope $S^i$ (in the $(d-1)$-dimensional affine space $\aff S^i$). For all $\eps>0$ as above $S^i$ remains the same up to a homothety. Let $1 \le k \le d-2$ be given. We want to construct a suitable polynomial $p_k$. For $i=1,\ldots,m$ let $p_{i,k-1}$ be a polynomial as in the assertion with respect to the polytope $S^i$ and the sets $D_{i,k-1}$ and $D_{i,k-2}$. Now let $q_{i,k}$ be the homogenization of $p_{i,k}$ with respect to the center $x_i$. Around $x_i$ the polynomial $q_{i,k}$ satisfies the properties a)-c). So, in order to generate $p_k$ we should combine the polynomials $q_{i,k}$ in a suitable way.

\emph{Step~4: Notations and Remarks.} For $k=0,\ldots,d-1$ let 
\begin{equation} \label{compactified playground}
	Q_{k-1} := R_k \cap (D_{k-1}(S) \setminus \intr (D_k(S)))
\end{equation}
Then $Q_{k-1}$ is compact. Let $G_1,\ldots,G_s$ be the facets of $D_k(S)$. (Here we keep $k$ fixed, so we omit the index $k$ at the $G_i$'s.) For $\delta>0$ and $x_i \in S_0 = \{x_1,\ldots,x_m\}$ let 
\begin{equation*}
 \mymark{D_k(S)}_{x_i,\delta} := \bigcap_{\overtwocond{j=1,\ldots,s}{x_i \not\in G_j}} \{l_{G_j} \ge \delta\}.
\end{equation*}
We have 
\begin{equation*}
 \mymark{D_k(S)}_{x_i,\delta} \subseteq \mymark{D_k(S)}_{x_i} \subseteq \cl \mymark{D_k(S)}_{x_i} = \bigcap_{\overtwocond{j=1,\ldots,s}{x_i \not\in G_j}} \{l_{G_j} \ge 0\}.
\end{equation*}
By Remark~\ref{rem:polytopes}f)ii) there exists a $\delta>0$ such that 
\begin{equation} \label{Q inclusion}
 Q_{k-1} \subseteq \bigcup_{i=1}^m \mymark{D_k(S)}_{x_i,\delta}.
\end{equation}

\emph{Step~5: Notations and Remarks.} For every decomposition $\{1,\ldots,m\} = \{i\} \cup A \cup B$ let 
\begin{align*}
 Q_{k-1}^+(A) & := \bigcap_{\alpha \in A} Q_{k-1} \cap \mymark{D_k(S)}_{x_\alpha}, \\
Q_{k-1}^-(B) & := \bigcap_{\beta \in B} Q_{k-1} \setminus \mymark{D_k(S)}_{x_\beta}, \\
Q_{k-1}^0(i) & := Q_{k-1} \cap \mymark{D_k(S)}_{x_i,\delta}, \\
Q_{k-1}(i,A,B) & := Q_{k-1}^0(i) \cap Q_{k-1}^+(A) \cap Q_{k-1}^-(B).
\end{align*}
The sets $Q_{k-1}^0(i)$ and $Q_{k-1}^-(B)$ are compact, the other two are in general not. In view of \eqref{Q inclusion}, we have  $Q_{k-1} = \bigcup_{(i,A,B)} Q_{k-1}(i,A,B)$, where the union runs over all partitions $\{1,\ldots,m\} = \{i\} \cup A \cup B$. Also, the sets $Q_{k-1}(i,A,B)$ form a partition of $Q_{k-1}(i)$ for a fixed $i$. 

\emph{Step~6: Construction of a polynomial, which satisfies the assertion on $Q_{k-1}$.} Again $\ext S = \{x_0,\ldots,x_m\}$, $G_1,\ldots,G_s$ are facets of $D_k(S)$ and $\delta$ is chosen as in Step~4. Moreover, by induction we constructed already polynomials $q_{i,k},$ $i=1,\ldots,m$. We choose $\rho>\delta$ such that $|l_{G_j}(x)| \le \rho$ for all $x \in R_k$, $j=1,\ldots,s$. Consider the polynomial $\kappa=\kappa_{\delta,\rho,m} \in \real[t]$ as in Lemma~\ref{lambda:mu:lem}.\ref{kappa:part}). We set
\begin{equation*}
 r_{i,k} := q_{i,k} \prod_{\overtwocond{j=1,\ldots,t}{x_i \not\in G_j}} (\kappa \circ l_{G_j})^n
\end{equation*}
and
\begin{equation*}
 r_k:= \sum_{i=1}^m r_{i,k},
\end{equation*}
where $n \in \natur$ is to be fixed later. We claim the following.

\emph{Step~7: Verification that $r_k$ is a local solution.} For sufficiently large $n$ one has:
\begin{enumerate}[a')]
 \item $r_k \ge 0$ on $S$,
 \item $r_k \le 0$ on $Q_{k-1}$,
	\item $\{r_k=0\} \cap Q_{k-1} \subseteq S$.
\end{enumerate}
Let us prove a'). We have $q_{i,k} \ge 0$ on $S$ for $i=1,\ldots,m$ and $\kappa \ge 0$ on $[-\rho,\rho]$, hence $r_{i,k} \ge 0$ on $S$ and finally $r_k\ge 0$ on $S$. In view of \eqref{Q inclusion} we may replace $Q_{k-1}$ by $Q_{k-1}(i,A,B)$ for a given decomposition $\{1,\ldots,m\} = \{i \} \cup A \cup B$. Clearly, we have $r_{i,k} \le 0$ on $Q_{k-1}^0(i)$ and $\{r_{i,k} = 0 \} \cap Q_{k-1}^0(i) \subseteq S$. Note that the factor $\prod_{\overtwocond{j=1,\ldots,t}{x_i \not\in G_j}} (\kappa \circ l_{G_j})^n$ is positive and arbitrarily large for sufficiently large $n$. Let 
	\begin{align*}
		r_{k,A} := \sum_{\alpha \in A} r_{\alpha,k} & & \mbox{and} & & r_{k,B} := \sum_{\beta \in B} r_{\beta,k}, 
	\end{align*}
so $r_k= r_{i,k} + r_{A,k} + r_{B,k}$. We have $r_{A,k} \le 0$ on $Q_{k-1}^0(i) \cap Q_{k-1}^+(A)$. So it is sufficient to show that $|r_{B,k}| \le \frac{1}{2} |r_{i,k}|$ on $Q_{k-1}^0(i) \cap Q_{k-1}^-(B)$. For this it is enough to show that $|r_{\beta,k}| \le \frac{1}{2m} |r_{i,k}|$ on $Q_{k-1}^0(i) \setminus \mymark{D_{k}(S)}_{x_{\beta}}$ for all $\beta \in B$. We write 
$$
	Q_{k-1}^0(i) \setminus \mymark{D_k(S)}_{x_\beta} = \bigcup_{\overtwocond{j=1,\ldots,t}{x_\beta \not\in G_j}} \left( Q_{k-1}^0(i) \cap \{l_{G_j} \le 0\}\right).
$$
So finally it remains to show that $|r_{\beta,k}| \le \frac{1}{2m} |r_{i,k}|$ on $Q_{k-1}^0(i) \cap \{l_{G_j} \le 0\}$ for given $\beta, j$ and all sufficiently large $n \in \natur$. Clearly $\{r_{i,k}=0\} \subseteq \{r_{\beta,k}=0\}$ and $q_{\beta,k}$ is bounded on this set. Moreover, by the properties of $\kappa$, we have 
$$\left| \prod_{\overtwocond{j=1,\ldots,t}{x_\beta \not\in G_j}} l_{G_j}\right| \le \frac{3}{4}.$$
Thus, the claim follows from the H\"ormander-{\L}ojasiewicz inequality (see Theorem~\ref{more than loj}).

\emph{Step 8: Conclusion.} Let $T_1 := Q_{k-1}$, $T_2 := D_{k-1}(S) \setminus \intr (R_k)$ and $h \in \real[X]$ be a non-negative polynomial which vanishes only on the Zariski closure of $S \cap T_1 = \ext_k S$. Such a polynomial can easily be found with the additional property that $h > 0$ on $T_2$ and for $x \in T_2$: $h(x) \rightarrow \infty$ as $\|x\|\rightarrow \infty$. Finally, let $f=r_k$. For $S, T_1, T_2, h, f$ we apply Proposition~\ref{S:T1:T2}, which gives us a polynomial $p_k$ such that $p_k \ge r_k$ on $S$, $\sign (p_k) \le \sign (r_k)$ on $Q_{k-1}$ and $p_k<0$ on $D_{k-1}(S) \setminus \intr (R_k)$. That means $p_k$ satisfies a)-c).
\end{proof}

As an immediate consequence of Proposition~\ref{main proposition} we obtain

\begin{corollary} \label{polytopes:result} 
	Let $S \subseteq \real^d$ be a $d$-dimensional polytope. Then there are polynomials $p_0,\ldots,p_{d-1} \in \real[X]$ such that $S = \{p_0 \ge 0, \ldots, p_{d-1} \ge 0\}.$ 
\end{corollary}

In the proof of Theorem~\ref{every polyhedron is representable} we shall use the following observation, see \cite[Proposition~2.1]{Groetschel-Henk-2003} and also \cite[Section~6.5]{MR1659509} or use an argument with fans as in \cite[Chapter VI, Section~7]{MR1393194}.

\begin{proposition} \label{number of vanishing poly}
 Let $S$ be a $d$-polyhedron in $\real^d$, let $S = \{q_1 \ge 0,\ldots,q_m \ge 0\}$ for $q_1,\ldots,q_m \in \real[X]$ and let $F$ be a face of $S.$ Then at least $d-\dim F$ polynomials $q_i$ vanish on $F$. In particular, $m \ge d- \dim F$.
\end{proposition}

\begin{proof}[Proof of Theorem~\ref{every polyhedron is representable}]
Most arguments presented below are also given in \cite{Averkov-Henk-2009b}, but for the sake of completeness we give the whole proof. 

We write $S=S_0 \times \real^k$, where $\real^d = \real^{d-k} \times \real^k$ and $S_0 \subseteq \real^{d-k}$ in suitable coordinates. We may assume that $k< d$, so $S$ admits at least one $k$-face, hence, by Proposition~\ref{number of vanishing poly}, $S$ cannot be described by fewer than $d-k$ polynomials.

Conversely, every polynomial description of $S_0$ easily extends to $S$. Note also, that $S_0$ does not contain any line. So it remains to show that a polyhedron $S \subseteq \real^d$ which does not contain any line is representable by $S=\{p_0 \ge 0,\ldots,p_{d-1} \ge 0\}$ for suitable $p_0,\ldots,p_{d-1} \in \real[X].$

For this we consider $\real^d$ as affine subspace $\setcond{x \in \real^{d+1}}{x_{d+1}=1}$ in $\real^{d+1}$ and form the cone $C^0$ over $S$, that is $C^0 := \setcond{\lambda a}{a \in S, \ \lambda \ge 0}$. The set $C:= \cl C^0$ is a polyhedral cone. Since $S$ does not contain a line, there is a linear form $l$ of $\real^{d+1}$ such that $C$ intersects the hyperplane $l=1$ properly, that is $S':=\{l=1\} \cap C$ is a polytope. By Corollary~\ref{polytopes:result} we can find polynomials $q_0,\ldots,q_{d-1}$ such that $S' = \{q_0 \ge 0,\ldots,q_{d-1} \ge 0, \ l=1\}$. Possibly after an appropriate modification with the help of $l$, we assume that $q_0,\ldots,q_{d-1}$ are homogeneous and of even degree. We would like that $\{q_0 \ge 0,\ldots,q_{d-1} \ge 0\} = C \cup (-C)$, but this is possibly not true, since $\{q_0 \ge 0,\ldots,q_{d-1} \ge 0, \ l=0\}$ may contain points other than $o$. In order to avoid the above situation we keep $q_1,\ldots,q_{d-1}$ as before and adjust $q_0$. We have $\{q_0=0\} \cap S = \ext S$ and $\{q_0 \ge 0, \ l=1\}$ is compact. Then $q_0$ is negative semidefinite on $\{l=0\}$. We need to replace $q_0$ by a polynomial $p$ with $\{p \ge 0, \ l=0\} = \{o\}$. For this let $r_1,\ldots,r_m$ be (homogeneous) non-negative quadratic polynomials such that $\{r_1=0\},\ldots,\{r_m=0\}$ are the affine hulls of the extremal rays of $C$, and let $r:=r_1,\ldots,r_m$. Now let 
$$
	p:=q_0 l^{N-\deg q_0} - c r^N,
$$
where $c>0$ is sufficiently small and $N \in \natur$ is sufficiently large. Then $\{p \ge 0, \ l=0\} = \{o\}$ and $S \subseteq \{p \ge 0, \ l=1\}$ by the H\"ormander-{\L}ojasiewicz inequality. Hence 
$$
	\{p\ge 0, \ q_1 \ge 0, \ldots,q_{d-1} \ge 0\} = C \cup (-C)
$$
Now we set $p_0(X_1,\ldots,X_d) := p(X_1,\ldots,X_d,1)$ and $p_i(X_1,\ldots,X_d) = q_i(X_1,\ldots,X_d,1)$ for $i=1,\ldots,d-1$.
\end{proof}

\begin{remark}
	From the proofs we see that the polynomial representation $S = \{p_0 \ge 0,\ldots, p_{d-1} \ge 0\}$ of a $d$-dimensional polytope $S$ that we construct above satisfies the condition 
	\begin{equation} \label{p_i ext_i}
		\{p_i = 0 \} \cap S \subseteq \ext_i S 	
	\end{equation}
   for $i \in \{0,\ldots,d-1\}$. In fact, an arbitrarily representation of this size can easily be converted into a form satisfying the above properties. Let $S = \{q_0 \ge 0,\ldots,q_{d-1} \ge 0\}$ be an arbitrary polynomial representation of a $d$-dimensional polyhedron $S$. For $p_i:=\sigma_{i+1}(q_0,\ldots,q_{d-1})$ one has $S = \{p_0 \ge 0,\ldots, p_{d-1} \ge 0\}$ (see Lemma~\ref{el:sym:lemma}). Furthermore, in view of Proposition~\ref{number of vanishing poly},  $p_i$'s satisfy \eqref{p_i ext_i}.
\end{remark}

\small
\providecommand{\bysame}{\leavevmode\hbox to3em{\hrulefill}\thinspace}
\providecommand{\MR}{\relax\ifhmode\unskip\space\fi MR }
\providecommand{\MRhref}[2]{%
  \href{http://www.ams.org/mathscinet-getitem?mr=#1}{#2}
}
\providecommand{\href}[2]{#2}

\end{document}